\documentclass[11pt]{article}
\usepackage{amsmath,amssymb,amsthm}
\usepackage{fancybox}
\usepackage{graphicx}
\usepackage{color}
\usepackage{array}
\usepackage{longtable}
\usepackage{colortab}
\usepackage{colortbl}
\usepackage{arydshln}

\allowdisplaybreaks 

\begin{document}

\def\fl#1{\left\lfloor#1\right\rfloor}
\def\cl#1{\left\lceil#1\right\rceil}
\def\ang#1{\left\langle#1\right\rangle}
\def\stf#1#2{\left[#1\atop#2\right]}
\def\sts#1#2{\left\{#1\atop#2\right\}}
\def\eul#1#2{\left\langle#1\atop#2\right\rangle}
\def\N{\mathbb N}
\def\Z{\mathbb Z}
\def\R{\mathbb R}
\def\C{\mathbb C}
\newcommand{\ctext}[1]{\raise0.2ex\hbox{\textcircled{\scriptsize{#1}}}}

\newtheorem{theorem}{Theorem}
\newtheorem{Prop}{Proposition}
\newtheorem{Cor}{Corollary}
\newtheorem{Lem}{Lemma}
\newtheorem{Sublem}{Sublemma}
\newtheorem{Def}{Definition}
\newtheorem{Conj}{Conjecture}

\newenvironment{Rem}{\begin{trivlist} \item[\hskip \labelsep{\it
Remark.}]\setlength{\parindent}{0pt}}{\end{trivlist}}

\title{Congruence properties of Lehmer-Euler numbers 
}

\author{
Takao Komatsu
\\
\small Faculty of Education\\[-0.8ex]
\small Nagasaki University\\[-0.8ex]
\small Nagasaki 852-8521 Japan\\[-0.8ex]
\small \texttt{komatsu@nagasaki-u.ac.jp}\\\\
Guo-Dong Liu\\
\small Department of Mathematics and Statistics\\[-0.8ex]
\small Huizhou University\\[-0.8ex]
\small Huizhou, Guangdong 516007 China\\[-0.8ex]
\small \texttt{lgd@hzu.edu.cn}
}

\date{
}
\maketitle

\begin{abstract}
Certain generalization of Euler numbers was defined in 1935 by Lehmer using cubic roots of unity, as a natural generalization of Bernoulli and Euler numbers. 
In this paper, Lehmer's generalized Euler numbers are studied to give certain congruence properties together with recurrence and explicit formulas of the numbers. We also show a new polynomial sequence and its properties. Some identities including Euler and central factorial numbers are obtained.  
\end{abstract}

\section{Introduction}

In 1935, D. H. Lehmer \cite{Lehmer} introduced and investigated generalized Euler numbers $W_n$, 
which are defined as
$$
\sum_{n=0}^\infty W_n\frac{t^n}{n!}=\frac{3}{e^{t}+e^{\omega t}+e^{\omega^2 t}}=\left(\sum_{l=0}^\infty\frac{t^{3l}}{(3l)!}\right)^{-1},
$$ 
where $\omega=\frac{-1+\sqrt{-3}}{2}$ and $\omega^2=\bar\omega=\frac{-1-\sqrt{-3}}{2}$ are the cubic roots of unity.  
Note that $W_n=0$ if and only if $3\nmid n$, and $(-1)^n W_{3 n}>0$.   
They can be considered as analogues of Euler numbers $E_n$ defined as 
$$
\sum_{n=0}^\infty E_n\frac{t^n}{n!}=\frac{2}{e^t+e^{-t}}=\frac{1}{\cosh t}=\left(\sum_{l=0}^\infty\frac{t^{2l}}{(2l)!}\right)^{-1}, 
$$  
and its subsequence of the absolute values of its non-zero elements $\{\mathcal W_{3 n}\}_{n\ge 0}$, where $\mathcal W_n:=(-1)^n W_{n}$, is recorded in \cite[A002115]{oeis}: 
{\small  
\begin{multline*} 
1, 1, 19, 1513, 315523, 136085041, 105261234643, 132705221399353,  \\
254604707462013571, 705927677520644167681, 2716778010767155313771539, \dots\,. 
\end{multline*} 
} 
We shall call $W_n$ the ($n$th) {\it Lehmer-Euler number} in this paper for short. 

In \cite{BK} several properties on determinants and representations of the Lehmer-Euler numbers are studied.  Some generalizations are shown in \cite{KP}.  
In this paper, we show several congruence properties together with recurrence and explicit formulas of the numbers.    

\section{Basic properties of $W_n$} 

In this section, we remember some fundamental properties of Lehmer-Euler numbers $W_n$ from \cite{BK,KP}. They are used to show congruence properties in the next main section.  

\begin{theorem}\label{thm:3property}
\begin{itemize}
\item[\rm (i)](Recurrence formula)
We have $W_0=1$ and
$$
W_{3 n}=-\sum_{k=0}^{n-1}\binom{3 n}{3 k}W_{3 k}\quad(n\ge 1).  
$$  
\label{prop:ex3}
\item[\rm (ii)](Explicit formula)
For $n\ge 1$, we have
$$
W_{3 n}=(3 n)!\sum_{k=1}^n(-1)^k\sum_{i_1+\cdots+i_k=n\atop i_1,\dots,i_k\ge 1}\frac{1}{(3 i_1)!\cdots(3 i_k)!}\,.  
$$ 
\label{th:ex1}
\item[\rm (iii)](Determinant expression)
For $n\ge 1$, we have
$$
W_{3n}=(-1)^{n}(3 n)!\left|
\begin{array}{ccccc}
\frac{1}{3!}&1&0&&\\
\frac{1}{6!}&\frac{1}{3!}&&&\\
\vdots&\vdots&\ddots&1&0\\
\frac{1}{(3 n-3)!}&\frac{1}{(3 n-6)!}&\cdots&\frac{1}{3!}&1\\
\frac{1}{(3 n)!}&\frac{1}{(3 n-3)!}&\cdots&\frac{1}{6!}&\frac{1}{3!}
\end{array}
\right|\,. 
$$ 
\label{th:det1} 
\end{itemize}
\end{theorem}

By using Trudi's formula (\cite{Brioschi,Muir,Trudi}) we can obtain another explicit expression and inversion relation of $W_n$. 

\begin{theorem}  
For $n\ge 1$, we have
\begin{equation*} 
W_{3n}=(3 n)!\sum_{t_1+2 t_2+\cdots+n t_n=n}\binom{t_1+\cdots+t_n}{t_1,\dots,t_n}
\prod_{l=1}^{n}\left(\frac{-1}{(3l)!}\right)^{t_l}\,. 
\end{equation*} 
\label{th:otherexp}
\end{theorem}

By using the inversion formula (see, e.g. \cite{KR}), we obtain the following.  

\begin{Prop}  
For $n\ge 1$, we have 
$$ 
\frac{(-1)^n}{(3 n)!}=\left|\begin{array}{ccccc}
\frac{W_3}{3!}&1&0&&\\
\frac{W_6}{6!}&\frac{W_3}{3!}&&&\\
\vdots&\vdots&\ddots&1&0\\
\frac{W_{3 n-3}}{(3 n-3)!}&\frac{W_{3 n-6}}{(3 n-6)!}&\cdots&\frac{W_3}{3!}&1\\
\frac{W_{3 n}}{(3 n)!}&\frac{W_{3 n-3}}{(3 n-3)!}&\cdots&\frac{W_6}{6!}&\frac{W_3}{3!}  
\end{array}\right|\,. 
$$ 
\label{th400}  
\end{Prop}

\section{Congruence properties}  

In this section, we give some congruence properties of $W_n$.  

\begin{theorem}  
For any non-negative integer $n$, we have 
$$
W_{3n}\in \mathbb{Z}\quad {\rm and}\quad 
W_{3 n}\equiv(-1)^n\pmod 9\,.
$$ 
\label{th:congmod3}
\end{theorem}

\begin{proof}  
By Theorem \ref{prop:ex3} (i), for $n>0$, we have 
$$ 
W_{3 n}=-\sum_{k=0}^{n-1}\binom{3 n}{3 k}W_{3 k}\,.  
$$  
Since $\binom{3 n}{3 k}$ is an integer and $W_0=1$, $W_{3n}$ is an integer for all nonnegative integer $n$.  
By  
$$
\sum_{k=0}^n\binom{3 n}{3 k}x^{3 k}=\frac{1}{3}\sum_{j=0}^2(1+\omega^j x)^{3 n}\,, 
$$ 
we get  
\begin{align}  
\sum_{k=0}^n\binom{3 n}{3 k}(-1)^k&=\frac{1}{3}\sum_{j=0}^2(1-\omega^j)^{3 n}
=3^{n-1}\bigl((-\sqrt{-3})^{n}+(\sqrt{-3})^{n}\bigr)\,. 
\label{eq:oh01}
\end{align}  
Hence, if $n$ is odd positive, then 
$$
\sum_{k=0}^n\binom{3 n}{3 k}(-1)^k=0\,, 
$$ 
and if $n=2 n'$ is even positive, then 
$$
\sum_{k=0}^n\binom{3 n}{3 k}(-1)^k=2(-1)^{n'}3^{3 n'-1}\equiv 0\pmod 9\,. 
$$ 
By induction with Theorem \ref{prop:ex3} (i), for $n>0$, we have   
$$   
W_{3 n}\equiv(-1)^n-\sum_{k=0}^{n}\binom{3 n}{3 k}(-1)^k \equiv(-1)^n\pmod 9\,. 
$$  
\end{proof}

From Theorem \ref{th:congmod3}, for $m\ge 0$, we have 
$$
W_{6 m}\equiv 1\quad\hbox{and}\quad W_{6 m+3}\equiv -1\pmod{3^2}\,.  
$$

For a prime $p$, the following Lucas's Theorem (\cite{Lucas}, see also \cite{Loveless}) is well-known. 

\begin{Lem}[Lucas's Theorem]  
If $p$ is a prime and $m$ and $n$ are non-negative integers with base $p$ representations $m=m_r p^r+\cdots+m_1 p+m_0$ and $n=n_r p^r+\cdots+n_1 p+n_0$, respectively, then 
$$
\binom{m}{n}=\binom{m_r p^r+\cdots+m_1 p+m_0}{n_r p^r+\cdots+n_1 p+n_0}\equiv\prod_{i=0}^r\binom{m_i}{n_i}\pmod p\,. 
$$ 
\label{Lucas}
\end{Lem}

By using Lemma \ref{Lucas}, we have the following result.

\begin{theorem} 
For $n\ge 0$, we have 
$$
W_{9 n}\equiv(-1)^n,\quad W_{9 n+3}=(-1)^{n-1},\quad W_{9 n+6}\equiv(-1)^{n-1}8\pmod{3^3}\,. 
$$
In particular, for $m\ge 0$, 
\begin{align*}
&W_{18 m}\equiv 1,\quad W_{18 m+3}\equiv -1,\quad W_{18 m+6}\equiv -8\,,\\
&W_{18 m+9}\equiv -1,\quad W_{18 m+12}\equiv 1,\quad W_{18 m+15}\equiv 8\pmod{3^3}\,. 
\end{align*} 
\label{th:9}
\end{theorem}  
\begin{proof} 
When $n=0$, we know that $W_0=1$, $W_3=-1$ and $W_6=19\equiv -8\pmod{27}$. 
By Lucas's Theorem, we have for $\ell=0,1,2$
$$
\binom{9 n}{9 k+3\ell}\equiv\binom{3 n}{3 k+\ell}\pmod 3\,. 
$$ 
In fact, these congruences hold for $\pmod{3^l}$ ($\forall l\ge 1$) if $n$ is enough large.   
Since 
$$
\sum_{k=0}^{n-1}\binom{3 n}{3 k+1}x^{3 k+1}=\frac{1}{3}\sum_{j=0}^2\omega^{2 j}(1+\omega^j)^{3 n}\,,
$$
by setting $x=-1$ we have 
\begin{align*}
\sum_{k=0}^{n-1}\binom{3 n}{3 k+1}(-1)^{k-1}&=\frac{1}{3}\sum_{j=0}^2\omega^{2 j}(1+\omega^j)^{3 n}\\
&=\frac{1}{3}\bigl(\omega^2(-3\sqrt{-3})^n+\omega(3\sqrt{-3})^n\bigr)\\
&=3^{n-1}\bigl(\omega^2(-\sqrt{-3})^n+\omega(\sqrt{-3})^n\bigr)\,.
\end{align*}
Hence,  
$$
\sum_{k=0}^{n-1}\binom{3 n}{3 k+1}(-1)^{k-1}=
\begin{cases}
(-1)^{\frac{n+1}{2}}3^{\frac{3 n-1}{2}}&\text{when $n$ is odd};\\ 
(-1)^{\frac{n}{2}}3^{\frac{3 n}{2}-1}&\text{when $n$ is even}.
\end{cases}
$$
Since 
$$ 
\sum_{k=0}^{n-1}\binom{3 n}{3 k+2}x^{3 k+2}=\frac{1}{3}\sum_{j=0}^2\omega^{j}(1+\omega^j)^{3 n}\,,
\label{eq:en1}
$$
we have 
\begin{align*}
\sum_{k=0}^{n-1}\binom{3 n}{3 k+2}(-1)^{k-1}8&=-\frac{8}{3}\sum_{j=0}^2\omega^{j}(1+\omega^j)^{3 n}\\
&=-8\cdot 3^{n-1}\bigl(\omega(-\sqrt{-3})^n+\omega^2(\sqrt{-3})^n\bigr)\,.
\end{align*}
Hence, 
$$
\sum_{k=0}^{n-1}\binom{3 n}{3 k+2}(-1)^{k-1}8=
\begin{cases}
(-1)^{\frac{n+1}{2}}8\cdot 3^{\frac{3 n-1}{2}}&\text{when $n$ is odd};\\  
(-1)^{\frac{n}{2}}8\cdot 3^{\frac{3 n}{2}-1}&\text{when $n$ is even}.
\end{cases}
$$
By Theorem \ref{prop:ex3} (i) together with the proof of Theorem \ref{th:congmod3}, for $n\ge 3$, we have 
\begin{align*} 
W_{9 n}&=-\sum_{k=0}^{3 n-1}\binom{9 n}{3 k}W_{3 k}\\
&=-\sum_{k=0}^{n-1}\binom{9 n}{9 k}W_{9 k}-\sum_{k=0}^{n-1}\binom{9 n}{9 k+3}W_{9 k+3}-\sum_{k=0}^{n-1}\binom{9 n}{9 k+6}W_{9 k+6}\\
&\equiv (-1)^{n}-0-0=(-1)^n\pmod{3^3}\,.  
\end{align*}
When $n=1$, we have $W_9=-1513\equiv -1\pmod{3^3}$. 
When $n=2$, we have $W_{18}=105261234643\equiv 1\pmod{3^3}$. 
In fact, for $n\ge n_0$ 
$$
W_{9 n}\equiv\begin{cases}
(-1)^n\pmod{3^{\frac{3 n_0-1}{2}}}&\text{if $n_0$ is odd};\\
(-1)^n\pmod{3^{\frac{3 n_0}{2}-1}}&\text{if $n_0$ is even}. 
\end{cases}
$$ 

Since 
$$
\sum_{k=0}^n\binom{3 n+1}{3 k}x^{3 k}=\frac{1}{3}\sum_{j=0}^2(1+\omega^j x)^{3 n+1}\,,
$$ 
for $x=-1$, we have 
\begin{align*}  
\sum_{k=0}^n\binom{3 n+1}{3 k}(-1)^{k}&=\frac{1}{3}\sum_{j=0}^2(1-\omega^j)^{3 n+1}\\
&=\frac{1}{3}\big((-3\sqrt{-3})^n(1-\omega)+(3\sqrt{-3})^n(1-\omega^2)\bigr)\,.
\end{align*} 
Hence, 
$$
\sum_{k=0}^n\binom{3 n+1}{3 k}(-1)^{k}=\begin{cases}
(-1)^{\frac{n+1}{2}}3^{\frac{3 n-1}{2}}&\text{$n$ is odd};\\
(-1)^{\frac{n}{2}}3^{\frac{3}{2}n}&\text{$n$ is even}.
\end{cases}
$$
Since 
$$
\sum_{k=0}^n\binom{3 n+1}{3 k+1}x^{3 k+1}=\frac{1}{3}\sum_{j=0}^2\omega^{2 j}(1+\omega^j x)^{3 n+1}\,,
$$ 
we have 
\begin{align*}  
\sum_{k=0}^n\binom{3 n+1}{3 k+1}(-1)^{k-1}&=\frac{1}{3}\sum_{j=0}^2\omega^{2 j}(1-\omega^j)^{3 n+1}\\
&=\frac{1}{3}\big(\omega^2(-3\sqrt{-3})^n(1-\omega)+\omega(3\sqrt{-3})^n(1-\omega^2)\bigr)\\
&=\frac{1}{3}\big((\omega^2-1)(-3\sqrt{-3})^n+(\omega-1)(3\sqrt{-3})^n\bigr)\,.
\end{align*} 
Hence, 
$$
\sum_{k=0}^n\binom{3 n+1}{3 k+1}(-1)^{k-1}=\begin{cases}
(-1)^{\frac{n-1}{2}}3^{\frac{3 n-1}{2}}&\text{$n$ is odd};\\
(-1)^{\frac{n}{2}+1}3^{\frac{3}{2}n}&\text{$n$ is even}.
\end{cases}
$$
Since 
$$
\sum_{k=0}^{n-1}\binom{3 n+1}{3 k+2}x^{3 k+2}=\frac{1}{3}\sum_{j=0}^2\omega^j(1+\omega^j x)^{3 n+1}\,,
$$ 
we have 
\begin{align*}  
\sum_{k=0}^{n-1}\binom{3 n+1}{3 k+2}(-1)^{k-1}8&=-\frac{8}{3}\sum_{j=0}^2\omega^j(1-\omega^j)^{3 n+1}\\
&=-\frac{8}{3}\big(\omega(-3\sqrt{-3})^n(1-\omega)+\omega^2(3\sqrt{-3})^n(1-\omega^2)\bigr)\\
&=-\frac{8}{3}\big((\omega-\omega^2)(-3\sqrt{-3})^n+(\omega^2-\omega)(3\sqrt{-3})^n\bigr)\,.
\end{align*} 
Hence, 
$$
\sum_{k=0}^{n-1}\binom{3 n+1}{3 k+2}(-1)^{k}=\begin{cases}
(-1)^{\frac{n+1}{2}}16\cdot 3^{\frac{3 n-1}{2}}&\text{$n$ is odd};\\
0&\text{$n$ is even}.
\end{cases}
$$
By Theorem \ref{prop:ex3} (i), for $n\ge 2$, we have 
\begin{align*} 
W_{9 n+3}&=-\sum_{k=0}^{3 n+1}\binom{9 n+3}{3 k}W_{3 k}\\
&=-\sum_{k=0}^{n}\binom{9 n+3}{9 k}W_{9 k}-\sum_{k=0}^{n-1}\binom{9 n+3}{9 k+3}W_{9 k+3}-\sum_{k=0}^{n-1}\binom{9 n+3}{9 k+6}W_{9 k+6}\\
&\equiv -0+(-1)^{n-1}-0
=(-1)^{n-1}\pmod{3^3}\,.  
\end{align*}
When $n=1$, we have $W_{12}=315523\equiv 1\pmod{3^3}$. 
In fact, for $n\ge n_0$ 
$$
W_{9 n+3}\equiv\begin{cases}
(-1)^{n-1}\pmod{3^{\frac{3 n_0+1}{2}}}&\text{if $n_0$ is odd};\\
(-1)^{n-1}\pmod{3^{\frac{3 n_0}{2}}}&\text{if $n_0$ is even}. 
\end{cases}
$$ 

Since 
$$
\sum_{k=0}^n\binom{3 n+2}{3 k}x^{3 k}=\frac{1}{3}\sum_{j=0}^2(1+\omega^j x)^{3 n+2}\,,
$$ 
for $x=-1$, we have 
\begin{align*}  
\sum_{k=0}^n\binom{3 n+2}{3 k}(-1)^{k}&=\frac{1}{3}\sum_{j=0}^2(1-\omega^j)^{3 n+2}\\
&=\frac{1}{3}\big((-3\sqrt{-3})^n(1-\omega)^2+(3\sqrt{-3})^n(1-\omega^2)^2\bigr)\\
&=-\big(\omega(-3\sqrt{-3})^n+\omega^2(3\sqrt{-3})^n\bigr)\,.
\end{align*} 
Hence, 
$$
\sum_{k=0}^n\binom{3 n+2}{3 k}(-1)^{k}=\begin{cases}
(-1)^{\frac{n+1}{2}}3^{\frac{3 n+1}{2}}&\text{$n$ is odd};\\
(-1)^{\frac{n}{2}}3^{\frac{3}{2}n}&\text{$n$ is even}.
\end{cases}
$$
Since 
$$
\sum_{k=0}^n\binom{3 n+2}{3 k+1}x^{3 k+1}=\frac{1}{3}\sum_{j=0}^2\omega^{2 j}(1+\omega^j x)^{3 n+2}\,,
$$ 
we have 
\begin{align*}  
\sum_{k=0}^n\binom{3 n+2}{3 k+1}(-1)^{k-1}&=\frac{1}{3}\sum_{j=0}^2\omega^{2 j}(1-\omega^j)^{3 n+2}\\
&=\frac{1}{3}\big(\omega^2(-3\sqrt{-3})^n(1-\omega)^2+\omega(3\sqrt{-3})^n(1-\omega^2)^2\bigr)\\
&=-\big((-3\sqrt{-3})^n+(3\sqrt{-3})^n\bigr)\,.
\end{align*} 
Hence, 
$$
\sum_{k=0}^n\binom{3 n+2}{3 k+1}(-1)^{k-1}=\begin{cases}
0&\text{$n$ is odd};\\
(-1)^{\frac{n}{2}-1}2\cdot 3^{\frac{3}{2}n}&\text{$n$ is even}.
\end{cases}
$$
Since 
$$
\sum_{k=0}^n\binom{3 n+2}{3 k+2}x^{3 k+2}=\frac{1}{3}\sum_{j=0}^2\omega^j(1+\omega^j x)^{3 n+2}\,,
$$ 
we have 
\begin{align*}  
\sum_{k=0}^n\binom{3 n+2}{3 k+2}(-1)^{k-1}8&=-\frac{8}{3}\sum_{j=0}^2\omega^j(1-\omega^j)^{3 n+2}\\
&=-\frac{8}{3}\big(\omega^2(-3\sqrt{-3})^n(1-\omega)^2+\omega(3\sqrt{-3})^n(1-\omega^2)^2\bigr)\\
&=8\big(\omega^2(-3\sqrt{-3})^n+\omega(3\sqrt{-3})^n\bigr)\,.
\end{align*} 
Hence, 
$$
\sum_{k=0}^n\binom{3 n+2}{3 k+2}(-1)^{k-1}8=\begin{cases}
(-1)^{\frac{n+1}{2}}8\cdot 3^{\frac{3 n+1}{2}}&\text{$n$ is odd};\\
(-1)^{\frac{n}{2}}8\cdot 3^{\frac{3}{2}n}&\text{$n$ is even}.
\end{cases}
$$
By Theorem \ref{prop:ex3} (i), for $n\ge 2$, we have 
\begin{align*} 
W_{9 n+6}&=-\sum_{k=0}^{3 n+1}\binom{9 n+6}{3 k}W_{3 k}\\
&=-\sum_{k=0}^{n}\binom{9 n+6}{9 k}W_{9 k}-\sum_{k=0}^{n}\binom{9 n+6}{9 k+3}W_{9 k+3}-\sum_{k=0}^{n-1}\binom{9 n+6}{9 k+6}W_{9 k+6}\\
&\equiv -0-0+(-1)^{n-1}8
=(-1)^{n-1}8\pmod{3^3}\,.  
\end{align*}
When $n=1$, we have $W_{15}=-136085041\equiv 8\pmod{3^3}$.
$$
W_{9 n+6}\equiv-1-24-3\equiv -1\pmod{3^3}\,. 
$$
In fact, for $n\ge n_0$ 
$$
W_{9 n+6}\equiv\begin{cases}
(-1)^{n-1}8\pmod{3^{\frac{3 n_0+1}{2}}}&\text{if $n_0$ is odd};\\
(-1)^{n-1}8\pmod{3^{\frac{3 n_0}{2}}}&\text{if $n_0$ is even}. 
\end{cases}
$$ 
\end{proof}


\subsection{Lehmer-Euler numbers modulo powers of three}  

A classical result on Euler numbers modulo power of two due to Stern \cite{Stern} asserts that 
\begin{equation}  
E_{2 n}\equiv E_{2 m}\pmod{2^k}\Longleftrightarrow 2 n\equiv 2 m\pmod{2^k}\,.
\label{eq:stern}   
\end{equation}    
In 1875 Stern \cite{Stern} gave a brief sketch of a proof of (\ref{eq:stern}), then Frobenius \cite{Frobenius} amplified Stern's sketch in 1910. In 1979 Ernvall \cite{Ernvall} said that he could not understand Frobenius' proof and gave his own proof using the umbral calculus. In 2000 an induction proof of (\ref{eq:stern}) was given by Wagstaff \cite{Wagstaff}. Sun \cite{Sun05} obtained an explicit congruence for Euler numbers modulo powers of two to give a new proof of (\ref{eq:stern}). Liu \cite{Liu08} gave a still new proof of (\ref{eq:stern}) by using the Euler numbers of order $2$ and the central factorial numbers.  

In the case of the Lemer-Euler numbers, the following is predicted.  
$$
W_{3 n}\equiv W_{3 m}\pmod{3^k}\Longleftrightarrow 3 n\equiv 3 m\pmod{3^k}\,.
$$ 
But unfortunately this does not hold. 
Together with Theorems \ref{th:congmod3} and \ref{th:9}, indeed, we have the following.  

\begin{Prop}
We have 
\begin{align*}
&\{W_{3 n}\pmod{3^2}\}_{n=0}^\infty=1,\overline{8,1}\quad\hbox{(period $2$)}\,,\\
&\{W_{3 n}\pmod{3^3}\}_{n=0}^\infty=1,\overline{\overleftrightarrow{26, 19, 26}, \overleftrightarrow{1, 8, 1}}\quad\hbox{(period $6$)}\,,\\
&\{W_{3 n}\pmod{3^4}\}_{n=0}^\infty\\
&=1,\overline{\overleftrightarrow{80, 19, 26, 28, 62, 28, 26, 19, 80}, \overleftrightarrow{1, 62, 55, 53, 19, 53, 55, 62, 1}}\quad\hbox{(period $2\cdot 3^2$)}\,,\\
&\{W_{3 n}\pmod{3^5}\}_{n=0}^\infty=1,\overleftarrow{242, 19, 188, 109, 62, 28, 107, 19, 80, 82, 224, 217, 53}, \\
&\quad\overrightarrow{181, 53, 217, 224, 82, 80, 19, 107, 28, 62, 109, 188, 19, 242},\\
&\quad\overleftarrow{1, 224, 55, 134, 181, 215, 136, 224, 163, 161, 19, 26, 190, 62, 190, 26, 19, 161, }\\
&\quad\overrightarrow{163, 224, 136, 215, 181, 134, 55, 224, 1}\quad\hbox{(period $2\cdot3^3$)}\,. 
\end{align*}
\label{prp:lec0}
\end{Prop} 

\noindent 
{\it Remark.}  
For $k\ge 6$, the palindromic property does not hold for $\{W_{3 n}\pmod{3^k}\}_{n=0}^\infty$.

Based on the above results, we predict that the following can be generally stated. 

\begin{Conj}  
If $3 n\equiv 3 m\pmod{2\cdot 3^k}$, then $W_{3 n}\equiv W_{3 m}\pmod{3^{k+1}}$.   
\label{conj:lec0}
\end{Conj}

\begin{proof}[Proof of Proposition \ref{prp:lec0}.]
Concerning the case of $W_{3 n}\pmod{3^4}$, we can show 
\begin{align*}
&W_{27 n}\equiv(-1)^n,\quad W_{27 n+3}\equiv(-1)^{n-1},\\
&W_{27 n+6}\equiv W_{27 n+24}\equiv(-1)^n 19,\quad W_{27 n+15}\equiv(-1)^{n-1}19,\\
&W_{27 n+9}\equiv W_{27 n+21}\equiv(-1)^n 26,\\ 
&W_{27 n+12}\equiv W_{27 n+18}\equiv(-1)^n 28\pmod{3^4}\,. 
\end{align*}
For example, similarly to the proof of Theorem \ref{th:9}, the first identity can be shown by Theorem \ref{prop:ex3} (i) and the facts of the following relations. 
\begin{align*}
&\sum_{k=0}^n\binom{27 n}{27 k}(-1)^k\equiv
\begin{cases}
0&\text{if $n=1$ or $n\ge 3$}\\
-18&\text{if $n=2$},
\end{cases}\\
&\sum_{k=0}^{n-1}\binom{27 n}{27 k+3}(-1)^{k-1}\equiv
\begin{cases}
-9&\text{if $n=1$}\\
0&\text{if $n\ge 2$},
\end{cases}\\
&\sum_{k=0}^n\binom{27 n}{27 k+6}(-1)^k 19\equiv
\begin{cases}
36&\text{if $n=1$}\\
0&\text{if $n\ge 2$},
\end{cases}\\
&\sum_{k=0}^{n-1}\binom{27 n}{27 k+9}(-1)^{k}26\equiv
\begin{cases}
-3&\text{if $n=1$}\\
9&\text{if $n=2$}\\
0&\text{if $n\ge 3$},
\end{cases}\\
&\sum_{k=0}^n\binom{27 n}{27 k+12}(-1)^k 28\equiv
\begin{cases}
45&\text{if $n=1$}\\
0&\text{if $n\ge 2$},
\end{cases}\\
&\sum_{k=0}^{n-1}\binom{27 n}{27 k+15}(-1)^{k-1}19\equiv
\begin{cases}
36&\text{if $n=1$}\\
0&\text{if $n\ge 2$},
\end{cases}\\
&\sum_{k=0}^n\binom{27 n}{27 k+18}(-1)^k 28\equiv
\begin{cases}
3&\text{if $n=1$}\\
9&\text{if $n=2$}\\
0&\text{if $n\ge 3$},
\end{cases}\\
&\sum_{k=0}^{n-1}\binom{27 n}{27 k+21}(-1)^{k}26\equiv
\begin{cases}
-36&\text{if $n=1$}\\
0&\text{if $n\ge 2$},
\end{cases}\\
&\sum_{k=0}^{n-1}\binom{27 n}{27 k+24}(-1)^{k}19\equiv
\begin{cases}
9&\text{if $n=1$}\\
0&\text{if $n\ge 2$}
\end{cases}
\pmod{3^4}\,.
\end{align*}
Other identities are similarly shown and their proofs are omitted.  
\end{proof}

\section{Incomplete Lehmer-Euler numbers} 

The various incomplete numbers have been introduced and studied (see, e.g., \cite{Ko16,KLM16,KR}).   
In order to generalize the numbers $W_n$, we shall introduce two kinds of incomplete Lehmer-Euler numbers.   
For $m\ge 1$, define the incomplete Lehmer-Euler numbers  
$W_{n,\le m}$ and $W_{n,\ge m}$ by 
\begin{equation}  
\sum_{n=0}^\infty W_{n,\le m}\frac{t^n}{n!}=\frac{1}{1+\sum_{l=1}^m\frac{t^{3 l}}{(3 l)!}}
\label{def:resteuler} 
\end{equation}  
and 
\begin{equation}  
\sum_{n=0}^\infty W_{n,\ge m}\frac{t^n}{n!}=\frac{1}{1+\sum_{l=m}^\infty\frac{t^{3 l}}{(3 l)!}}\,,
\label{def:assoeuler} 
\end{equation}
respectively. 
When $m\to\infty$, $W_n=W_{n,\le\infty}$, and 
when $m=1$, $W_n=W_{n,\ge 1}$. 
Hence, both incomplete numbers are reduced to the original Lehmer-Euler numbers. 
Moreover, note that $W_{n,\le m}=W_{n,\ge m}=0$ if and only if $3\nmid n$.  

These incomplete Lehmer-Euler numbers satisfy the following recurrence and explicit formulas and determinant expressions. 
The proposition can be proven similar to Theorem \ref{thm:3property}, 
thus we omit the proof. 

\begin{Prop}
\begin{itemize}
\item[\rm (i)](Recurrence formulas)
We have the following.

\noindent{\rm (a)} 
$$
W_{3 n,\le m}=-\sum_{k=\max\{n-m,0\}}^{n-1}\binom{3 n}{3 k}W_{3 k,\le m}\quad(n\ge 1) 
$$
with $W_{0,\le m}=1$. 
\label{prop:ex3r}

\noindent{\rm (b)} 
$$
W_{3 n,\ge m}=-\sum_{k=0}^{n-m}\binom{3 n}{3 k}W_{3 k,\ge m}\quad(n\ge m) 
$$
with $W_{0,\ge m}=1$ and $W_{3,\ge m}=\cdots=W_{3 m-3,\ge m}=0$.   
\label{prop:ex3a} 
\item[\rm (ii)](Explicit formulas)
For $n,m\ge 1$,  we have the following:

\noindent{\rm (a)}
$$  
W_{3 n,\le m}=(3 n)!\sum_{k=1}^n(-1)^k\sum_{i_1+\cdots+i_k=n\atop 1\le i_1,\dots,i_k\le m}\frac{1}{(3 i_1)!\cdots(3 i_k)!}\,.  
$$
\label{th:ex11r}

\noindent{\rm (b)}
$$ 
W_{3 n,\ge m}=(3 n)!\sum_{k=1}^n(-1)^k\sum_{i_1+\cdots+i_k=n\atop i_1,\dots,i_k\ge m}\frac{1}{(3 i_1)!\cdots(3 i_k)!}\,.  
$$ 
\label{th:ex11a}
\item[\rm (iii)](Determinant expressions)
For $n\ge m\ge 1$, we have the following:

\noindent{\rm (a)}
$$
W_{3 n,\le m}=(-1)^n(3 n)!\left|\begin{array}{cc} 
\underbrace{ 
\begin{array}{ccc}
\frac{1}{3!}&1&0\\
\vdots&&\ddots\\
\frac{1}{(3 m)!}&&\\
0&\ddots&\\ 
&&\ddots\\
&&0
\end{array}
}_{n-m} 
\underbrace{ 
\begin{array}{ccc}
&&\\
\ddots&&\\
&\ddots&\\
&\ddots&0\\
&&1\\
\frac{1}{(3 m)!}&\cdots&\frac{1}{3!}
\end{array}
}_{m} 
\end{array} 
\right|\,. 
$$
\label{th:rleu}

\noindent{\rm (b)}
$$
W_{3 n,\ge m}=(-1)^n(3 n)!\left|\begin{array}{cc} 
\underbrace{ 
\begin{array}{ccc}
0&1&0\\
\vdots&&\ddots\\
0&&\\
\frac{1}{(3 m)!}&\ddots&\\ 
\vdots&&\ddots\\
\frac{1}{(3 n)!}&\cdots&\frac{1}{(3 m)!}
\end{array}
}_{n-m+1} 
\underbrace{ 
\begin{array}{ccc}
&&\\
\ddots&&\\
&\ddots&\\
&\ddots&0\\
&&1\\
0&\cdots&0
\end{array}
}_{m-1} 
\end{array} 
\right|\,. 
$$
\end{itemize}
\end{Prop}

The other explicit formulas, due to Trudi's formula, are also 
obtained for the incomplete cases, and the proof is similar 
to Theorem \ref{th:otherexp}, thus we omit the proof.

\begin{Prop}  
For $n\ge m\ge 1$, we have 
\begin{align*}
W_{3n,\le m}&=(3 n)!\sum_{t_1+2 t_2+\cdots+m t_m=n}\binom{t_1+\cdots+t_m}{t_1,\dots,t_m}
\prod_{l=1}^{m}\left(\frac{-1}{(3l)!}\right)^{t_l}\\
{\it and}\hspace{7mm}& &\\
W_{3n,\ge m}&=(3 n)!\sum_{m t_m+(m+1)t_{m+1}+\cdots+n t_n=n}\binom{t_m+t_{m+1}+\cdots+t_n}{t_m,t_{m+1},\dots,t_n}
\prod_{l=m}^{n}\left(\frac{-1}{(3l)!}\right)^{t_l}\,. 
\end{align*}
\label{th300}  
\end{Prop}

\section{Higher order Lehmer-Euler numbers}

For positive integers $r$ and $\alpha$, 
the {\it generalized Lehmer-Euler numbers} of order $\alpha$, which is denoted by $W_{r,n}^{(\alpha)}$ is defined by 
\begin{equation}  
\sum_{n=0}^\infty W_{r,n}^{(\alpha)}\frac{t^n}{n!}=\left(\frac{r}{\sum_{j=0}^{r-1}e^{\zeta^j t}}\right)^\alpha=\left(\sum_{l=0}^\infty\frac{t^{r l}}{(r l)!}\right)^{-\alpha}\,, 
\label{def:hle}
\end{equation}
where $\zeta=\zeta_r$ is the primitive $r$-th root of unity.  
Hence, when $r=3$ and $\alpha=1$, $W_n=W_{1,n}^{(1)}$ is the original Lehmer-Euler number.  
Since 
$$
(1+w)^{-\alpha}=\sum_{l=0}^\infty\binom{\alpha+l-1}{l}(-w)^l\quad(|w|<1)\,, 
$$
putting 
$$
1+w=\frac{1}{r}\sum_{j=0}^{r-1}e^{\zeta^j t}\,, 
$$
we have 
\begin{align*}
&W_{r,n}^{(\alpha)}
=\left.\frac{d^n}{d t^n}\left(\frac{r}{\sum_{j=0}^{r-1}e^{\zeta^j t}}\right)^\alpha\right|_{t=0}\\
&=\sum_{l=0}^n\binom{\alpha+l-1}{l}\left.\frac{d^n}{d t^n}\left(1-\frac{1}{r}\sum_{j=0}^{r-1}e^{\zeta^j t}\right)^l\right|_{t=0}\\
&=\sum_{l=0}^n\binom{\alpha+l-1}{l}\left.\frac{d^n}{d t^n}\sum_{k=0}^l\binom{l}{k}\left(-\frac{1}{r}\sum_{j=0}^{r-1}e^{\zeta^j t}\right)^k\right|_{t=0}\\
&=\sum_{l=0}^n\binom{\alpha+l-1}{l}\sum_{k=0}^l\left(-\frac{1}{r}\right)^k\binom{l}{k}\\
&\qquad\times\left.\frac{d^n}{d t^n}\sum_{i_1+i_2+\cdots+i_r=k\atop i_1,i_2,\dots,i_r\ge 0}\frac{k!}{i_1!i_2!\dots i_r}\prod_{j=1}^r e^{\zeta^{j-1}i_j t}\right|_{t=0}\\
&=\sum_{k=0}^n\left(-\frac{1}{r}\right)^k\sum_{i_1+i_2+\cdots+i_r=k\atop i_1,i_2,\dots,i_r\ge 0}\frac{k!}{i_1!i_2!\dots i_r!}\left(\sum_{j=1}^r\zeta^{j-1}i_j\right)^n\\
&\qquad\times\binom{\alpha+l-1}{l}\binom{l}{k}\\
&=\sum_{k=0}^n\left(-\frac{1}{r}\right)^k\binom{\alpha+k-1}{k}\sum_{i_1+i_2+\cdots+i_r=k\atop i_1,i_2,\dots,i_r\ge 0}\frac{k!}{i_1!i_2!\dots i_r!}\left(\sum_{j=1}^r\zeta^{j-1}i_j\right)^n\\
&\qquad\times\sum_{l=0}^{n-k}\binom{\alpha+k+l-1}{l}\\
&=\sum_{k=0}^n\left(-\frac{1}{r}\right)^k\binom{\alpha+k-1}{k}\binom{\alpha+n}{n-k}\sum_{i_1+i_2+\cdots+i_r=k\atop i_1,i_2,\dots,i_r\ge 0}\frac{k!}{i_1!i_2!\dots i_r!}\left(\sum_{j=1}^r\zeta^{j-1}i_j\right)^n\,.
\end{align*}

When $\alpha=1$, the generalized Lehmer-Euler numbers $W_{r,n}$ are reduced as 
$$
W_{r,n}=\sum_{k=0}^n\left(-\frac{1}{r}\right)^k\binom{n+1}{n-k}\sum_{i_1+i_2+\cdots+i_r=k\atop i_1,i_2,\dots,i_r\ge 0}\frac{k!}{i_1!i_2!\dots i_r!}\left(\sum_{j=1}^r\zeta^{j-1}i_j\right)^n\,.
$$
If $r=3$, we have an expression of the Lehmer-Euler numbers of higher-order: 
$$
W_n=\sum_{k=0}^n\left(-\frac{1}{3}\right)^k\binom{\alpha+n}{n-k}\binom{\alpha+k-1}{k}\sum_{i_1+i_2+i_3=k\atop i_1,i_2,i_3\ge 0}\frac{k!}{i_1!i_2!i_3!}\left(i_1+i_2\omega+i_3\omega^2\right)^n\,.
$$
If $r=2$, we have an expression of the Euler numbers of higher-order (\cite{Luo05}):  
\begin{equation}
E_n^{(\alpha)}=\sum_{k=0}^n\left(-\frac{1}{2}\right)^k\binom{\alpha+n}{n-k}\binom{\alpha+k-1}{k}\sum_{j=0}^k\binom{k}{j}(k-2 j)^n\,. 
\label{eq:luo}
\end{equation}

\section{A new polynomial sequence and its properties}   

For a complex number $x$ and a non-negative integer $k$, the polynomial sequence $\Delta(x,k)$ is defined by 
\begin{equation}
\Delta(x,k+1)=(x+1)(2 x+1)\Delta(x+1,k)-x^2\Delta(x,k)\quad(k\ge 0)
\label{def:polyseq}
\end{equation}
with $\Delta(x,0)=1$. 
By using the recurrence relation in (\ref{def:polyseq}), we have 
\begin{align*} 
\Delta(x,1)&=x^2+3 x+1\,,\\
\Delta(x,2)&=x^4+10 x^3+25 x^2+20 x+5\,,\\
\Delta(x,3)&=x^6+21 x^5+140 x^4+385 x^3+490 x^2+287 x+61\,.\\
\end{align*}
The central factorial numbers $t(n,k)$ of the first kind (also known as the Stirling numbers of the first kind with level $2$ \cite{KRV1}) are given by the formula 
\begin{equation}  
x\left(x+\frac{n}{2}-1\right)\left(x+\frac{n}{2}-2\right)\cdots\left(x-\frac{n}{2}+1\right)=\sum_{k=0}^\infty t(n,k)x^k
\label{def:cfn1}
\end{equation} 
and its exponential generating function is given by 
\begin{equation}   
\left(2\log\left(\frac{x}{2}+\sqrt{\frac{x^2}{4}+1}\right)\right)^k=k!\sum_{n=k}^\infty t(n,k)\frac{x^n}{n!}\,.
\label{gf:cfn1}
\end{equation} 
(see, e.g., \cite{Riordan}). It follows from (\ref{def:cfn1}) and (\ref{gf:cfn1}) that 
\begin{equation}  
t(n,k)=t(n-2,k-2)-\frac{(n-2)^2}{4}t(n-2,k)
\label{rec:cfn1}
\end{equation}
with $t(n,0)=0$ ($n>0$), $t(n,n)=1$ ($n\ge 0$) and 
$t(n,k)=0$ ($n+k$ is odd, $k>n$ or $k<0$).  

The central factorial numbers $T(n,k)$ of the second kind (also known as the Stirling numbers of the second kind with level $2$ \cite{KRV2}) are given by the formula 
\begin{equation}  
x^n=\sum_{k=0}^\infty T(n,k)x\left(x+\frac{k}{2}-1\right)\left(x+\frac{k}{2}-2\right)\cdots\left(x-\frac{k}{2}+1\right)
\label{def:cfn2}
\end{equation} 
and its exponential generating function is given by 
\begin{equation}   
(e^{\frac{x}{2}}-e^{-\frac{x}{2}})^k=k!\sum_{n=k}^\infty T(n,k)\frac{x^n}{n!}\,.
\label{gf:cfn2}
\end{equation} 
(see, e.g., \cite{Riordan}). It follows from (\ref{def:cfn2}) and (\ref{gf:cfn2}) that 
\begin{equation}  
T(n,k)=T(n-2,k-2)+\frac{k^2}{4}T(n-2,k)
\label{rec:cfn2}
\end{equation}
with $T(n,0)=0$ ($n>0$), $T(n,n)=1$ ($n\ge 0$) and 
$T(n,k)=0$ ($n+k$ is odd, $k>n$ or $k<0$).  

In this section, we show the identity between Euler numbers and the central factorial numbers of the second kind in terms of the polynomial sequence $\Delta(x,k)$ in (\ref{def:polyseq}). 

\begin{theorem}
For non-negative integers $n$ and $k$, we have 
$$
E_{2 n+2 k}=\sum_{j=0}^n\frac{(-1)^{j-k}(2 j)!\Delta(j,k)}{2^j}T(2 n,2 j)\,. 
$$
\label{th:eu-cfn2}
\end{theorem}

\noindent 
{\it Remark.}  
When $k=0$, the identity in Theorem \ref{th:eu-cfn2} is reduced to 
\begin{equation}
E_{2 n}=\sum_{j=0}^n\frac{(-1)^{j}(2 j)!}{2^j}T(2 n,2 j)\,. 
\label{eq:lz}
\end{equation}
(see \cite{LZ}).  
When $n=0$, it is reduced to 
\begin{equation}
E_{2 k}=(-1)^k\Delta(0,k)\,. 
\label{eq:e2k}
\end{equation}

The identity involving the central factorial numbers of the first kind and the polynomial sequence $\Delta(x,k)$ can be given as follows.  

\begin{theorem}
For a non-negative integer $n$, we have 
$$
\sum_{j=0}^n(-4)^{n-j}t(2 n+1,2 j+1)\Delta(x,j)=(x+1)(x+2)\cdots(x+2 n)\,. 
$$
\label{th:eu-cfn1}
\end{theorem} 

\noindent 
{\it Remark.}  
When $x=0$, by (\ref{eq:e2k}), it is reduced to 
$$
\sum_{j=0}^n 4^{n-j}t(2 n+1,2 j+1)E_{2 j}=(-1)^n(2 n)!\,. 
$$

\begin{proof}[Proof of Theorem \ref{th:eu-cfn2}.] 
We shall prove by induction. 
When $k=0$, the identity is valid because of (\ref{eq:lz}) in \cite{LZ}. 
Suppose that the identity is valid for some $k$. By (\ref{def:polyseq}) and (\ref{rec:cfn2}), we have 
\begin{align*}
&E_{2 n+2(k+1)}=E_{2(n+1)+2 k}\\
&=\sum_{j=0}^{n+1}\frac{(-1)^{j-k}(2 j)!\Delta(j,k)}{2^j}T(2 n+2,2 j)\\
&=\sum_{j=0}^{n+1}\frac{(-1)^{j-k}(2 j)!\Delta(j,k)}{2^j}\bigl(T(2 n,2 j-2)+j^2 T(2 n,2 j)\bigr)\\
&=\sum_{j=0}^{n}\frac{(-1)^{j-k+1}(2 j+2)!\Delta(j+1,k)}{2^{j+1}}T(2 n,2 j)\\
&\qquad +\sum_{j=0}^{n}\frac{(-1)^{j-k}(2 j)!j^2\Delta(j,k)}{2^j}T(2 n,2 j)\\
&=\sum_{j=0}^{n}\frac{(-1)^{j-k+1}(2 j)!}{2^{j}}\bigl((j+1)(2 j+1)\Delta(j+1,k)-j^2\Delta(j,k)\bigr)T(2 n,2 j)\\
&=\sum_{j=0}^{n}\frac{(-1)^{j-(k+1)}(2 j)!\Delta(j,k+1)}{2^{j}}T(2 n,2 j)\,.
\end{align*}
Hence, the identity is valid for $k+1$. 
\end{proof} 

\begin{proof}[Proof of Theorem \ref{th:eu-cfn1}.] 
When $n=0$, by $t(1,1)\Delta(x,0)=1$, the identity is valid.  
Suppose that the identity is valid for some $n$. By (\ref{def:polyseq}) and (\ref{rec:cfn1}), we have 
\begin{align*}
&\sum_{j=0}^{n+1}(-4)^{n-j+1}t(2 n+3,2 j+1)\Delta(x,j)\\
&=\sum_{j=0}^{n+1}(-4)^{n-j+1}\left(t(2 n+1,2 j-1)-\frac{(2 n+1)^2}{4}t(2 n+1,2 j+1)\right)\Delta(x,j)\\
&=\sum_{j=0}^{n+1}(-4)^{n-j+1}t(2 n+1,2 j-1)\Delta(x,j)\\
&\qquad +\sum_{j=0}^{n}(-4)^{n-j}(2 n+1)^2 t(2 n+1,2 j+1)\Delta(x,j)\\
&=\sum_{j=0}^{n}(-4)^{n-j}t(2 n+1,2 j+1)\Delta(x,j+1)\\
&\qquad +(2 n+1)^2(x+1)(x+2)\cdots(x+2 n)\\
&=\sum_{j=0}^{n}(-4)^{n-j}t(2 n+1,2 j+1)\bigl((x+1)(2 x+1)\Delta(x+1,j)-x^2\Delta(x,j)\bigr)\\
&\qquad +(2 n+1)^2(x+1)(x+2)\cdots(x+2 n)\\
&=(x+1)(2 x+1)(x+2)(x+3)\cdots(x+2 n+1)\\
&\qquad -x^2(x+1)(x+2)\cdots(x+2 n)
+(2 n+1)^2(x+1)(x+2)\cdots(x+2 n)\\
&=(x+1)(x+2)\cdots(x+2 n+1)(x+2 n+2)\,. 
\end{align*}
Hence, the identity is valid for $k+1$. 
\end{proof}



\end{document}